\documentclass[11pt]{amsart}
\usepackage[utf8x]{inputenc}

\usepackage{dsfont}
\usepackage{amsmath}
\usepackage{amssymb}
\usepackage{graphicx}
\usepackage{amsfonts}
\usepackage{soul}
\usepackage{mathpazo}
\usepackage{color}
\usepackage{yfonts}
\usepackage{paralist}
\usepackage{stmaryrd}
\usepackage{amsxtra}
\usepackage{latexsym,mathrsfs}
\usepackage{algpseudocode}
\usepackage{algorithm}
\usepackage{mathabx}
\usepackage[a4paper, margin=1in]{geometry}
\newtheorem{theorem}{Theorem}
\newtheorem{lemma}{Lemma}
\newtheorem{Prop}{Proposition}
\newtheorem{Cor}{Corollary}

\title{On the Color Discrepancy of Spanning Trees in Random and Randomly Perturbed Graphs}
\author{Wenchong Chen}
\address[Chen]{Nankai University, Weijin Road 94, Nankai District, Tianjin, 300192, P. R. China}
\email{2212161@mail.nankai.edu.cn}

\author{Xiao-Chuan Liu}
\address[Liu]{Departamento de Matemática,
 Universidade Federal de Pernambuco,
	Avenida Jornalista Aníbal Fernandes - Cidade Universitária, Recife, Brasil}
\email{xiaochuan.liu@ufpe.br}

\author{Xu Yang}
\address[Yang]{Instituto de Computação, Universidade Federal de Alagoas,
	Av. Lourival Melo Mota, S/N, Maceió, Brasil} 
\email{yang@ic.ufal.br}
\begin{document}

\maketitle

\begin{abstract}
In this work, we study the color discrepancy of spanning trees in random graphs. We show that for the Erd\H{o}s-Rényi random graph $G(n,p)$ with $p$ above the connectivity threshold, the following holds with high probability: in every 2-edge-coloring of the graph, there exists a spanning tree with a linear number of leaves such that one color class contains more than $\frac{1 + \varepsilon}{2}n $ of the tree's edges. Here, $\varepsilon>0$ is a small absolute constant independent of $p$. 

We also extend this line of research to randomly perturbed dense graphs, showing that adding a few random edges to a dense graph typically creates a spanning tree with a large color discrepancy under any 2-edge-coloring. 
\end{abstract}

\section{Introduction}

The \emph{color discrepancy} problems ask to what extent one color can dominate another in an edge--colored graph. 
Among the various formulations, the hypergraph framework provides a natural general setting. 
Formally, given a hypergraph $\mathcal{H} = (X, \mathcal{E})$, 
consider a (not necessarily proper) two--coloring 
$f \colon X \to \{+1,-1\}$ 
of its vertex set. 
The \emph{discrepancy} of $f$ on $\mathcal{H}$ is 
$\mathcal{D}(\mathcal{H}, f)
   = \max_{A \in \mathcal{E}}
     \bigl| \sum_{x \in A} f(x) \bigr|$. 
The \emph{discrepancy} of the hypergraph $\mathcal{H}$ itself is then
\begin{equation}
\mathcal{D}(\mathcal{H})
   = \min_{f}
     \mathcal{D}(\mathcal{H}, f).
\end{equation}
For historical studies on this topic, see~\cite{beck1987irregularities, alexander2017geometric}.

For a given graph \(G\), we define an associated hypergraph \(\mathcal H = (X,\mathcal E)\), 
where the vertex set \(X = E(G)\) consists of the edges of \(G\), 
and each hyperedge in \(\mathcal E\) corresponds to a prescribed global substructure of \(G\). 
This formulation captures the inherent concentration of colours within large or 
structurally constrained subgraphs of \(G\). 
Typical examples of such global substructures include Hamilton cycles and spanning trees.
For example, the recent work by Balogh, Csaba, Jing, and Pluhár~\cite{balogh2020discrepancies} initiated the study of discrepancy problems in arbitrary graphs, focusing on the discrepancy of spanning trees and Hamiltonian cycles. This line of research was later extended by Gishboliner, Krivelevich, and Michaeli~\cite{gishboliner2022discrepancies} to $r$-edge colorings, with further developments in~\cite{hollom2024discrepancies}.

In this paper we study the random graph $G \sim G(n,p)$, and 
consider the hypergraph $\mathcal{H} = (X,\mathcal{E})$ defined by 
$X = E(G)$ and 
$\mathcal{E} = \mathcal{T}_G$, 
where $\mathcal{T}_G$ denotes the family of all spanning trees of $G$ having linearly many leaves. 
The associated color discrepancy problem captures the extent to which a two--coloring of the edges of $G$ 
can remain balanced across such spanning trees. 
Let \( n \) be a positive integer and \( p \in [0, 1] \) a real number. 
We denote by \( G(n, p) \) the \emph{binomial random graph}: the probability space consisting of all simple labeled graphs on \( n \) vertices, where each pair of distinct vertices is independently connected by an edge with probability \( p \). An event \( A \) in this probability space is said to occur \emph{with high probability} ({\bf whp}) if \( \mathbb{P}(A) \to 1 \) as \( n \to \infty \). Throughout the paper, we assume that 
$n$ is sufficiently large. In the context of random graphs, Gishboliner, Krivelevich, and Michaeli~\cite{gishboliner2022color} proved the following 
theorem:
\begin{theorem}[Theorem 1.5 in~\cite{gishboliner2022color}]
    Let $r\geq 2$ be an integer and let $\varepsilon>0$. Then there exist $C$ and $K$ such that if $p\geq C/n$, the random graph $G\sim G(n,p)$ is {\bf whp} such that in any $r$-coloring of its edges there exists a path of length at least $(\frac{2}{r+1}-o(1))n$ in which all but at most $K$ of the edges are of the same color. 
\end{theorem}

This theorem implies that, {\bf whp}, random graphs contain large tree-like substructures with significant color imbalance. However, the structures guaranteed by this result—long paths or trees of large diameter—do not address the case of trees with many leaves and small diameter.

In this work, we strengthen this perspective by focusing on spanning trees with additional structural constraints. Specifically, we focus on spanning trees with a large number of leaves, which inherently tend to have small diameters. Spanning trees with many leaves often correspond to topologies that enable efficient broadcasting or routing with minimal delay, a desirable property in distributed systems and communication networks. Such trees are of particular interest due to their relevance in practical scenarios such as network design, communication protocols, and models of real-world systems like small-world networks~\cite{eppstein1996spanning, ravi1996spanning}. Our results show that, even under the structural constraint of having linearly many leaves, random graphs still contain spanning trees with high color discrepancy, thereby complementing and extending the previous theory.

Let \( G \) be a graph with \( n \) vertices, and let  \( c(G) \) denote the number of its connected components.  
Define \( \mathcal{F} \) to be the collection of all spanning forests of \( G \) with exactly \( c(G) \) components, then let $\mathcal F_\alpha$ subcollection of forests that contain at least $\alpha n$ leaves:
\begin{align}\label{spanning_forests}
\mathcal F & = \mathcal{F}_G 
= \left\{ F : F \text{ is a spanning forest  of } G \text{ and } c(F) = c(G) \right\}.\\
\mathcal F_\alpha &=\{F\in \mathcal F : F \text{ has at least } \alpha n \text{ leaves}\}.
\end{align}
 We first obtain a high discrepancy result for spanning forests in random graphs $G\sim G(n,p)$, where $p=\frac{C}{n}$. 
\begin{theorem}\label{disc_forest}
Let \( \alpha \in (0,\frac 15) \) be fixed.  There exist constants \( \varepsilon > 0 \) and \( C > 0 \) such that the following holds. Let \( G \sim G(n, p) \) with \( p \geq  \frac{C}{n} \). With high probability, $G$ satisfies the following property:  
for any edge coloring \( \chi : E(G) \to \{+1, -1\} \), there exists a spanning forest \( F \in \mathcal{F}_\alpha \) such that
\begin{equation}
\left| \sum_{e \in E(F)} \chi(e) \right| \geq \varepsilon n.
\end{equation}
\end{theorem}

Based on Theorem~\ref{disc_forest}, we obtain our first main result. 

\begin{theorem}\label{main1}
Let \( p = \frac{\log n + \omega(1)}{n} \).  
Then with high probability, the random graph \( G(n, p) \) has the following property:  
for every 2-edge-coloring of \( G \), there exists a spanning tree with a linear number of leaves, in which more than 
\( \frac{1 + \varepsilon}{2}n \) of its edges are assigned the same color, for some constant \( \varepsilon > 0 \).
\end{theorem}

Beyond the pure Erd\H{o}s–Rényi model, we also study the color discrepancy of randomly perturbed dense graphs, which are obtained by adding a sparse set of random edges to an arbitrary dense graph. Formally, we define our model as the union \(G_{\alpha} \cup G(n,p)\), where \(G_{\alpha}\) is a dense \(n\)-vertex graph with minimum degree \(\delta(G_{\alpha}) \ge \alpha n\), and \(G(n,p)\) is an independent random graph with \(1/n^{2} \ll p \ll 1/n\). This model is relevant because it captures realistic scenarios where a well-structured network is slightly perturbed by random connections, such as in social networks, communication systems, or distributed infrastructures. Studying the color discrepancy of all spanning trees in this setting allows us to understand whether high discrepancy persists under minimal randomness, and whether even sparse perturbations are sufficient to guarantee strong imbalance in any 2-edge-coloring. These results highlight the robustness of color discrepancy phenomena beyond purely random graphs and their relevance in practical networks with sparse randomness. We next show that even a sparse random perturbation suffices to make any dense base graph not only connected but 3-connected, ensuring the discrepancy results extend to this setting.

\begin{theorem}\label{classical-perturb}
For any $\alpha\in (0,1)$, there exists a constant $\lambda=\lambda(\alpha)$ such that the following statement holds. For any graph $G_{\alpha}$ on $n$ vertices with $\delta(G_{\alpha})\geq \alpha n$, and a probability function $p=p(n)$ with $\frac{1}{n^2}\ll p(n)\ll\frac 1n$, with high probability, 
the randomly perturbed graph $H=G_{\alpha}\cup G(n,p)$ contains a spanning tree $T$ with 
$$\left|\sum_{e\in E(T)}\chi(e) \right|\geq\lambda pn^2$$
under any coloring $\chi:E(H)\rightarrow\{+1,-1\}$.
\end{theorem}
We remark that the above estimate is sharp in the following sense.  
Let \(G^\ast\) be a graph on \(n\) vertices consisting of two connected components \(A\) and \(B\), each of order roughly \(n/2\), and with minimum degree \(\delta(G^\ast)\ge \alpha n\).  
Consider \(H = G_{\alpha}^\ast \cup G(n,p)\).  
We colour all edges inside \(A\) red, and all edges inside \(B\) as well as those between \(A\) and \(B\) blue.  
Since, with high probability, the number of edges between \(A\) and \(B\) satisfies $e (G[A,B]) = (1+o(1))\tfrac{n^2p}{4} \ll n$,
any spanning tree of \(H\) has colour discrepancy at most \(O(n^2p)\).

Our results in Theorem~\ref{disc_forest}, Theorem~\ref{main1} and Theorem~\ref{classical-perturb} naturally extend to the case of 
$r$-edge-colorings, where $r \ge 2$ is fixed. In this setting, each edge of $G$ 
is assigned one of $r$ colors, and the goal is to find a spanning forest 
$F \in \mathcal{F}_\alpha$ whose color imbalance---measured as the maximum over all 
colors of the absolute deviation from the average---is large. The arguments used 
in the two-color case can be adapted with only minor modifications.

The paper is organized as follows. Section~2 contains the proofs of Theorems~\ref{disc_forest} and~\ref{main1} for the Erd\H{o}s–Rényi random graph, while Section~3 presents the proof of Theorem~\ref{classical-perturb} on randomly perturbed dense graphs.

\section{Proof of Theorem~\ref{disc_forest} and Theorem~\ref{main1}}

In this section, we present a series of auxiliary results that form the foundation for the proof of Theorem~\ref{disc_forest} and Theorem~\ref{main1}. We begin with several lemmas and propositions, followed by their proofs.

\begin{lemma}\label{random_graph}
Fix \( \delta > 0 \). Then for \( C \) sufficiently large (depending on \( \delta \)), the random graph \( G \sim G(n, p) \) with \( p = C/n \) satisfies the following properties with high probability:

\begin{enumerate}
    \item For any disjoint sets \( V_1, V_2 \subseteq V(G) \) with \( |V_1|, |V_2| \geq \delta n \), the induced bipartite subgraph on \( V_1 \cup V_2 \) contains a matching of size at least \( \frac 18 \delta n \).
    
    \item The giant component of \( G \) has size at least \( (1 - \rho)n \), where \( \rho \ll \delta \).
\end{enumerate}
\end{lemma}

\begin{proof}
Both parts are standard. For (1), note that the expected number of edges in the bipartite subgraph induced on \( V_1 \cup V_2 \) is \( C\delta^2 n \). Then, by a Chernoff bound for the lower tail, with probability at least \( 1 - e^{-\Omega(C\delta^2 n)} \), the number of edges is at least at least half its expectation, that is,
\[
\mathbb{P}\left[\text{number of edges} \geq \frac{1}{2} C \delta^2 n \right] \geq 1 - e^{-\Omega(C \delta^2 n)}.
\]
On the other hand, for any vertex $v\in V_1\cup V_2$, the degree follows $d(v)\sim \text{Bin}(\delta n, p)$. Use the Chernoff upper tail bound, for any \( t > 0 \),
\[
\mathbb{P}\big[ d(v) > (1 + t) \mathbb{E}[d(v)] \big] 
\leq \exp\left( - \frac{t^2}{2 + t} \cdot \mathbb{E}[d(v)] \right).
\]
Here the expected degree is 
\[\mathbb{E}[d(v)]=\delta n\cdot \frac{C}{n}=C\delta\]
Choosing \( t = 1 \) (which corresponds to bounding the probability that \( d(v) \) exceeds \( 2 C \delta \)), we obtain
\[
\mathbb{P}\big[ d(v) > 2 C \delta \big] 
\leq \exp\left( - \frac{C \delta}{3} \right).
\]
Use a union bound over all vertices in $V_1\cup V_2$. There are $2\delta n$ vertices in $V_1\cup V_2$, so the probability that any vertex has degree exceeding $2C\delta$ is at most $2\delta n\cdot \text{exp}(-\frac{C\delta}{3})$. As long as $C\delta\gg \log n$, this becomes $\mathbb{P}[\Delta>2C\delta]\leq o(1)$. So {\bf whp}, the maximum degree in the bipartite graph is at most \( \Delta \leq 2C\delta \). Therefore, between $V_1,V_2$, there exists a matching of size at least
$\lfloor\frac{ \frac{1}{2} C \delta^2 n }{2\Delta}\rfloor  \geq \frac{ \frac{1}{2} C \delta^2 n }{4C\delta} = \frac{1}{8} \delta n.$  
The total number of ways to choose subsets $V_1$, $V_2\subseteq V(G)$ is at most 
\begin{equation}
    {n \choose \delta n}\cdot{n-\delta n \choose \delta n}\leq \big(\frac{e}{\delta}\big)^{2\delta n}.
\end{equation}
Then the probability of the existence of a pair of $V_1$ and $V_2$ such that there are less than $\frac{1}{2}C\delta^2n$ edges between them is at most 
\begin{equation}
    \big(\frac{e}{\delta}\big)^{2\delta n}\cdot e^{-\Omega(C\delta^2n)}.
\end{equation}
When \( C\delta^2 n \gg \delta n \log \tfrac{1}{\delta} \), which is implied by \( C \gg \tfrac{\log \tfrac{1}{\delta}}{\delta} \), the union bound ensures that, {\bf whp}, the desired property holds for all choices of \( V_1 \) and \( V_2 \). Note that this failure probability tends to zero as \( n \to \infty \).

For item (2), note that the giant component of $G(n,p)$ 
has size $(1-\rho)n$, where $\rho$ is the unique solution in the interval $(0,1)$ for $\rho=e^{-C(1-\rho)}$. 
\end{proof}
For any fixed tree $T$, denote by $L_k(T)$ (for $k\in \mathbb{N}$) 
the set of vertices of $T$ with degree $k$. Finally, define the inner tree $I(T)$ the subgraph induced by the non-leaves.

\begin{Prop}\label{many_leaves}
For every \( \delta > 0 \), there exists a constant \( C = C(\delta) > 0 \) such that, if \( p = \frac{C}{n} \), then with high probability
 the giant component of \( G \sim G(n, p) \) contains a spanning tree \( T \) with $|L_1(T)| \geq \big( \frac{1}{3} - \delta \big ) n.$
\end{Prop}

\begin{proof}
First note that we may consider the union model $G_1\cup G_2$, where $G_i\sim G(n, p_1)$, and \( p_i = \frac{C_i}{n} \), $i=1,2$. For $G_1$, by item (2) of Lemma~\ref{random_graph}, 
when $C_1$ is sufficient large, the giant component of $G_1$ has size $(1-\rho)n$, where $\rho\leq \delta/1000$. For $G_2$, by Theorem 1.4 in~\cite{frieze2024introduction}, it suffices to prove the corresponding result for the model \( G(n, m) \) instead of working directly with \( G(n, p_2) \), 
with \( m = Cn \), where \( C \) is a sufficiently large constant.

Let \( T \) be a spanning tree of the giant component of $G_1$ with $n'=(1-\rho)n$ vertices. Define $\delta'$ to be such that 
$(\frac 13-\delta)=(\frac 13-\delta')(1-\rho)$. Then we partition the vertex set into degree classes \( L_k = L_k(T) \), and define its inner tree \( I(T) \). 

Suppose that $|L_1| < ( \frac{1}{3} - \delta')n'$;
otherwise, the conclusion already holds and there is nothing to prove.

\noindent We \textbf{claim} that 
$|L_2 \setminus L_1(I(T))|\geq 3\delta' n'$. 
\begin{proof}[Proof of the Claim]
    Indeed, we have 
that 
\begin{equation}
    \sum_{k\geq 3} k|L_k|+2 |L_2| +|L_1|=2n'-2<2n'.
\end{equation}
Together with $\sum_{k\geq 1}|L_k|=n'$, we obtain that 
\begin{equation}
    \sum_{k \geq 3}(k-2) |L_k|
    < |L_1|,
\end{equation}
and hence 
\begin{equation}
    \sum_{k \geq 3} |L_k|
    < |L_1|.
\end{equation}
Then it follows that 
$|L_2|+2|L_1|>n'$, and therefore 
$|L_2|>(\frac 13 +2\delta') n'$. 
Note that each vertex in $L_1(I(T))$ is adjacent to at least a vertex in $L_1$ and therefore
$|L_1(I(T))|\leq |L_1|< ( \frac{1}{3} - \delta')n'$. Hence $|L_2\setminus L_1(I(T))|\geq 3\delta' n'$. 
\end{proof}
Next, let \( u^* \in L_2 \setminus L_1(I(T)) \), and observe that both neighbors of \( u^* \) are non-leaves. Let \( u_1 \in I(T) \) be a vertex not adjacent to any neighbor of $u^*$ in $T$. Note that removing $u^*$ from $T$ disconnects the tree into two components, and one of the neighbors of $u^*$ lies in a different component from $u_1$. We  denote this neighbor as $u_2$.   
The pair $\{u_1 u_2\}$ can be used to construct a new tree with one additional leaf, namely,
$$T' = \left( T \setminus \{ u_2 u^* \} \right) \cup \{ u_1 u_2 \}.$$  
We call such a pair $\{u_1u_2\}$ a \emph{producer}.

Let us now describe an algorithm to iteratively add leaves to the tree 
$T$.

\begin{algorithm}[H]
\caption{Leaf-Increasing Algorithm}
\begin{algorithmic}[1]
\State \textbf{Initialize} \( T^{(0)} := T \), and set \( m := 0 \).
\While{ \( |L_1(T^{(m)})| < (\frac 13 -\delta') n' \)}
\State randomly choose an edge $e_m$ to \( T^{(m)} \)
    \If{\( e_m = u_1 u_2 \) is a producer edge for \( T^{(m)} \)}
    \State Update: \(
       T^{(m+1)} := \left( T^{(m)} \setminus \{ u_2 u^* \} \right) \cup \{ e_m \}, \ m:=m+1.
        \)
    \Else       
    \State Set: 
      \(
     T^{(m+1)} := T^{(m)},\  m:=m+1.
      \)
    \EndIf
    \EndWhile
\end{algorithmic}
\end{algorithm}

Note that, when the algorithm stops, we can update \( T = T^{(m)} \).

\noindent We \emph{Claim} that, {\bf whp}, the above algorithm stops when \( m \leq C_3 n \), for some $C_3$ to be determined later.
\begin{proof}[proof of the Claim]
Let us define, for each \( 1 \leq j \leq C_3 n \), the indicator random variable:
\begin{equation} 
Y_j = 
\begin{cases}
1, & \text{if } e_j \text{ is a producer or the algorithm terminates at time } m \leq j, \\
0, & \text{otherwise}.
\end{cases}
\end{equation}

Observe that whenever the \textbf{Leaf-Increasing Algorithm} does not terminate, we have
$|L_1(T^{(m)})| < ( \frac{1}{3} - \delta' )n'$ and $|L_2(T^{(m)}) \setminus L_1(I(T^{(m)})) | \geq 3\delta' n'$. 
Note that there are at least $3\delta' n'$ valid choices for a vertex~$u^*$. Moreover, the tree $T^{(m)}$ contains at least $( \frac{2}{3} + \delta')n'$ non-leaf vertices, each of which cannot join both of the two neighbors of $u^*$ in $T$. Therefore we can find at least $3\delta'(\frac 23+\delta')n'^2$ choices of the producers. 

Hence,
\begin{equation}
    \mathbb{E}\left[Y_j \mid Y_1, \dots, Y_{j-1}\right] \geq \frac{3\delta'(\frac 23+\delta')n'^2}{\frac{n'(n'-1)}{2}}\geq 4\delta'.
\end{equation}

If at time $m=C_4n'$ with $C_4 =\frac{1}{4\delta'}$, the algorithm has not terminated yet, we must have
\begin{equation}
    \sum_{j=1}^{C_4 n'} Y_j < \big(\frac 13 - \delta'\big) n'.
\end{equation}

By taking 
$X \sim \text{Bin}(C_4 n', 4\delta')$,
\begin{equation} 
\mathbb{P} \big[ \sum_{j=1}^{C_4 n'} Y_j < (\frac 13-\delta')n'\big]
\leq \mathbb{P} [ X < (\frac 13-\delta')n' ]
\leq e^{-\Omega(n')} = o(1).
\end{equation}
The first inequality follows from Lemma 28.24 of ~\cite{frieze2024introduction} and the second inequality follows by the Chernoff bound. 
\end{proof}

Hence, {\bf whp}, we can obtain \( T \) with at least $(\frac 13-\delta')n'=(\frac 13-\delta)n$ leaves. 
\end{proof}
For simplicity, we denote $L(T)=L_1(T)$ the set of leaves of $T$. Note that $L(I(T))$ is therefore the set of 
leaves of $I(T)$. 
\begin{Cor}\label{L}
For any fixed \( \delta > 0 \), there exists a constant \( C = C(\delta) \) such that the random graph \( G \sim G(n, C/n) \) with high probability contains a tree \( T \) satisfying
\begin{equation}
|L(T) \cup L(I(T))| \geq ( \tfrac{5}{9} - \delta )n.
\end{equation}
\end{Cor}
\begin{proof} By Proposition \ref{many_leaves}, there is a constant \( C \), such that {\bf whp} the random graph \( G(n, C/n) \) contains a tree \( T_1 \) with 
$|L(T_1)| \geq ( \frac{1}{3} - \frac{\delta}{2})n$. 
Then, within the inner tree \( I(T_1) \), we apply the {\bf Leaf-Increasing Algorithm} in the proof of Proposition~\ref{many_leaves} to a new independent copy of \( G(n, p) \) to update \( I(T_1) \). Suppose $T_2$ is a tree obtained by modifying $I(T_1)$. Note that $L(I(T_1))\subseteq L(T_2)$. With high probability, this results in
$|L(T_2)| \geq ( \frac{1}{3} - \frac{\delta}{2} ) ( \frac{2}{3} + \frac{\delta}{2})(1 - \rho)n$, where \( \rho \ll \delta \). Now, define a new tree $T$ by attaching each leaf $\ell\in L(T_1)$ back to the same vertex of $L(T_2)$ to which it was attached in the original tree $T_1$. Moreover, \( T \) satisfies
$|L(T) \cup L(I(T))|\geq |L(T_1)|+|L(T_2)| \geq \left( \tfrac{5}{9} - \delta \right)n.$ 
\end{proof}

Now we are ready to present the proof of Theorem~\ref{disc_forest}. 
\begin{proof}[Proof of Theorem~\ref{disc_forest}]

Let $\delta$ be a small constant such that $\frac{5}{9}-5\delta>\frac{1}{2}$. Choose $C$ sufficiently large (depending on $\delta$) and consider a random graph $G\sim G(n,p)$ with $p=C/n$. Fix an arbitrary edge colouring $\chi: E(G)\rightarrow \{-1,+1\}$. By Corollary~\ref{many_leaves}, we obtain a tree $T$ such that 
\begin{equation}
   |L(T)\cup L(I(T))|\geq (\frac{5}{9}-\delta)n.
\end{equation}
Fix a constant $\varepsilon \ll \delta$.  
We may assume that the discrepancy of \( T \) is strictly less than \( \varepsilon n \); otherwise, the theorem follows trivially. Let 
\begin{equation}
    M_1=L(T)\cup N_T(L(T)) \text{ and } M_2=L(I(T))\cup N_{I(T)}(L(I(T))).
\end{equation} 

\noindent {\bf Case 1. } Suppose $|L(I(T))|<2\delta n$.  Then \begin{equation}
    |L(T)|\geq \Big(\frac{5}{9}-3\delta \Big)n>\frac{1}{2}n+2\delta n.
\end{equation} 
Within the subgraph $T[M_1]$, there are at least $\delta n$ edges colored $+1$ and at least $\delta n$ edges colored $-1$. Let $E_1$ (respectively $E_2$) denote the set of edges in $T[M_1]$  colored $+1$ (respectively $-1$).  For each edge $e\in E_1\cup E_2$, let $v=v_e$ denote the leaf vertex incident to $e$ in $T$. Define 
\begin{equation}
    V_1=\{v_e : e\in E_1\} \text{ and } V_2=\{v_e : e\in E_2\}.  
\end{equation}
By Lemma~\ref{random_graph}, the bipartite subgraph induced on $G[V_1\cup V_2]$ contains a matching of size at least $\frac 18 \delta n$. 
Therefore, at least $\frac{1}{16}\delta n$ of these matching edges are assigned the same color, say $+1$. For each such edge $e=uv$ in the matching, where $u\in V_2$, we modify the tree $T$ by replacing the edge $e_u$ incident to $u$ in $E_2$ (previously colored $-1$) with the edge $uv$, obtaining   
\begin{equation}
    T:=(T\setminus \{e_u\}) \cup \{uv\}. 
\end{equation}
Clearly, after this modification, the new tree has discrepancy at least 
\begin{equation}
    \Big(\frac {1}{32}\delta -\varepsilon\Big)n. 
\end{equation}

\noindent {\bf Case 2. } Suppose $|L(I(T))|\geq 2\delta n$. If  $T[M_1]$ contains at least $\delta n$ edges colored $+1$ and at least $\delta n$ edges colored $-1$, then argument proceeds exactly as in the previous case. 

Now assume $|L(I(T))|\geq 2\delta n$ and that there are more than $\delta n$ edges in $T[M_1]$ colored with one of the colors, say $+1$. Note that the total number of edges in $T[M_1\cup M_2]$ is at least $(\frac{5}{9}-\delta)n$, which implies there are at least $3\delta n$ edges of each color in this subgraph. From $E(T[M_1])$, arbitrarily choose $\delta n$ edges of color $+1$ and denote this set by $E_1$. Observe that the edges in $E_1$ are adjacent to at most $\delta n$ edges of color $-1$ in $E(T[M_2])$. Then, in the subgraph $T[(M_1\cup M_2)\setminus V(E_1) ]$, there are still at least $\delta n$ edges colored $-1$. We randomly choose $\delta n$ such edges and denote this set by $E_2$. By construction, $V(E_1)\cap V(E_2)=\emptyset$. For each edge $e\in E_1\cup E_2$, define $v_2$ as follows:
\begin{enumerate}
    \item if $e\in E(T[M_1])$, let $v_e$  be the leaf vertex in $T$ incident to $e$; 
    \item if $e\in E(T[M_2])$, let $v_e$ be the leaf vertex in $L(I(T))$ incident to $e$.
\end{enumerate}
 Let 
 \begin{equation}
     V_1=\{v_e : e\in E_1\}, \ V_2=\{v_e : e\in E_2\}.
 \end{equation} 
 As in the previous case, by Lemma \ref{random_graph}, the bipartite subgraph induced on $G[V_1\cup V_2]$ contains a matching of size at least $\frac 18 \delta n$. If more edges in the matching are colored $+1$, we update the tree by adding the +1-colored matching edges and removing corresponding $-1$-colored edges from $E_2$.  Specifically, for each such edge $e=uv$ with $u\in V_2$, define 
 $$T:=(T\setminus \{e_u\}) \cup \{uv\},$$ where $e_u\in E_2$ is the edge originally colored $-1$ and incident to $u$. If instead more matching edges are colored $-1$, we perform the analogous operation:  add the $-1$-colored matching edges and remove the corresponding $+1$-colored edges from $E_1$. The same argument applied in either case. 

Finally, note that we can take $\varepsilon \ll \delta$. Moreover, the number of vertices outside of the giant component is at most $\rho n$, with $\rho\ll \delta$. The conclusion 
follows immediately. 
\end{proof}

\begin{proof} [Proof of Theorem~\ref{main1}]
Let \( p_1 = \frac{\log n + \omega(1)}{n} \) and \( p_2 = \frac{C}{n} \), where \( C > 0 \) is a sufficiently large constant. Define the graph \( G(n, p) = G(n, p_1) \cup G(n, p_2) \). Then \( p = \frac{\log n + \omega(1)}{n} \), as desired.  
Note that \( G(n, p_1) \) is connected with high probability.
Moreover, Theorem~\ref{disc_forest} implies that, {\bf whp}, \(G(n, p_2) \) has the property that for any 2-edge-coloring, one can find  a spanning forest
$F$ with $\Omega(n)$ leaves such that the difference between the number of edges of the two colors in $F$ is at least \( 2\varepsilon n \).
Now, take such a spanning forest \( F \) from \( G(n, p_2) \), and for each component, add a connecting edge from \( G(n, p_1) \) (which is connected).  This yields a spanning tree \( T \) in \( G(n, p) \) with a linear number of leaves and at most \( \delta n \) additional edges.

Hence, the discrepancy for the tree $T$ remains large. Indeed, more than \( \frac{1 + \varepsilon - \rho}{2}n \) edges of $T$ are assigned the same color.
This completes the proof.
\end{proof}

\section{proof of Theorem~\ref{classical-perturb}}
The proof of Theorem~\ref{classical-perturb} relies on the following result which was put forward in ~\cite{gishboliner2022discrepancies}. 
Following~\cite{gishboliner2022discrepancies}, for any graph $G$, let us denote by $s(G)$ 
the minimum integer $s$ such that the vertex set $V=V(G)$ admits a partition $V=V_1\cup V_2\cup S$, where $|V_1|=|V_2|$, $E(V_1,V_2)=\emptyset$, and $|S|=s$. 

\begin{theorem}\label{equipartition}[Theorem 1.1 in~~\cite{gishboliner2022discrepancies}]
There exists a constant $C>0$ such that any $3$-connected graph $G$ with any coloring $\chi:E(G)\rightarrow \{+1,-1\}$ contains a spanning tree $T$ such that
$$\left|\sum_{e\in E(T)}\chi(e)\right|\geq Cs(G).$$
\end{theorem}

Now we prove 
a structural lemma as follows.
\begin{lemma}\label{strong-connect}
For any $d=d(n)=o(n),\, \alpha\in(0,1)$ and graph $G$ on $n$ vertices with 
$\delta(G)\geq \alpha n$, there exists $U\subset V(G)$ such that $|U|=o(n)$ and each component of $G[V(G)\backslash U]$ is $d$-connected and contains at least $\frac{\alpha}{2}n$ vertices.  
\end{lemma}

\begin{proof}
Initialize $U = \emptyset$. While there exists a connected component $G_1$ of $G[V(G)\setminus U]$ and a set $W \subset V(G_1)$ with $|W| \le d$ such that $G_1 \setminus W$ is disconnected, add all vertices of $W$ to $U$. We claim that no component created during this process has fewer than $\alpha n/2$ vertices. Suppose $G'$ is the first component that violates this. At the previous step, all components had size at least $\alpha n/2$, so there are at most $2/\alpha$ components. Thus $|U| \le (\frac{2}{\alpha}+1)d$, and for any $v \in V(G')$ we have $d(v) \le |V(G')| + |U| \le \frac{\alpha n}{2} + (\frac{2}{\alpha}+1)d < \alpha n$, contradicting $\delta(G) \ge \alpha n$. Therefore, the procedure stops with each component $d$-connected and of size at least $\alpha n/2$, and $|U| = o(n)$.
\end{proof}

Intuitively, adding a small number of random edges to a sufficiently dense graph can significantly improve its connectivity, even if the original graph is not connected or not highly connected. In particular, a sparse random perturbation typically eliminates small vertex cuts and “bridges” between components, resulting in a graph that is highly robust. The following lemma formalizes this idea for 3-connectivity:

\begin{lemma}\label{3-connected}
    For any $\alpha\in (0,1)$, let $G_{\alpha}$ be a graph on $n$ vertices with $\delta(G_{\alpha})\geq \alpha n$, and a probability function $p=p(n)$ with $\frac{1}{n^2}\ll p(n)$. Then with high probability, the randomly perturbed graph $H=G_{\alpha}\cup G(n,p)$ is 3-connected.
\end{lemma}

\begin{proof}
Let us show that {\bf whp} $H$ is $3$-connected, provided that $p=\frac{f(n)}{n^2}$, where $f(n)=\omega(1)$.

First we show that there exists a vertex  partition $V(G_\alpha)=V_1\cup\cdots\cup V_L$ such that
every $|V_j|\ge (\alpha-\varepsilon)n$ for some fixed small constant $\varepsilon$,
and moreover $G_\alpha[V_j]$ is $3$-connected.

We first handle articulation vertices. If \(G_\alpha\) has a cut vertex \(u\), then \(G_\alpha-u\) splits into several connected components; rather than deleting \(u\) permanently, we treat each resulting component together with \(u\) as a separate part by assigning \(u\) to the component in which it has the most neighbors. Thus we obtain an initial vertex-disjoint partition into parts that are 2-connected blocks, and it suffices to apply the separation-pair refinement inside each 2-connected part.  

Now take any 2-connected part \(C\) that is not \(3\)-connected and choose a minimal separation pair \(\{v_1,v_2\}\) of \(C\). 
Choose a pair $\{v_1,v_2\}$ that forms a vertex cut of $C$. 
Note that each connected component of $C-\{v_1,v_2\}$ has size at least $\alpha n - 2$, 
since every vertex in $G_\alpha$ has degree at least $\alpha n$. 
For each vertex $v_i$ in this cut, we assign each cut vertex $v_i$ to exactly one of the resulting components, specifically, to the components 
in which it has $\Omega(n)$ neighbors. Proceed inductively: whenever we  encounter a part $C$ such that the induced subgraph 
$G_\alpha[C]$ is not $3$-connected, we find a vertex cut $\{w_1,w_2\}$ of $C$ and apply the same assignment rule. 
Then, in each connected component of $C-\{w_1,w_2\}$, every vertex retains degree 
at least $\alpha n - O(1)$ in $G_\alpha$, since only a constant number of vertices are removed at each step. In particular, if some component has size less 
than $(2\alpha-\varepsilon)n$,  it must already be $3$-connected, since otherwise another $2$-vertex cut would produce a smaller piece that violates the minimum degree condition. Hence, the process must terminate after $O(1)$ steps, yielding a vertex partition $V(G_{\alpha})=V_1\cup\dots\cup V_L$, where each $G_{\alpha}[V_i]$ is 3-connected and $|V_i|=\Omega(n)$.

For each 
pair $(V_i,V_{j})$, $i\neq j$, where both $|V_i|$ and $|V_{j}|$ 
have sizes at least $cn$. Let $X$ be the number of 3-matchings between two vertex classes $V_i$ and $V_j$.
The expected number of such $3$-matchings is 
$$\mathbb{E}[X]=3!\binom{|V_i|}{3}\binom{|V_{j}|}{3}p^3=\Theta(n^6)\cdot \left(\frac{f(n)}{n^2}\right)^3=\Theta\big(f(n)^{3}\big).$$ 
A straightforward second-moment (or Janson) calculation shows that the variance is $O(f(n)^5)$, and hence by Chebyshev
\[
\mathbb{P}[X=0] \le \frac{\mathrm{Var}(X)}{(\mathbb{E}X)^2} = O\big(\frac{f(n)^5}{f(n)^6}\big)=O\big(\frac{1}{f(n)}\big).
\]
Consequently, by a union bound over the constantly many part-pairs, {\bf whp} every pair admits a 3-matching.

A union bound for all $j=1,\dots,L-1$ shows that, {\bf whp}, $G(n,p)$ provides a $3$-matching for every pair $(V_j,V_{j+1})$. This implies that $H$ is $3$-connected.
\end{proof}

Now we are ready to Prove Theorem~\ref{classical-perturb}.

\begin{proof}[Proof of Theorem~\ref{classical-perturb}]
Let $C_1$ be the constant $C$ taken in Theorem~\ref{equipartition}.  
We will present the proof for $\lambda=\frac{C_1\alpha^2}{32}$.
Given $G_{\alpha}$ with $\delta (G_{\alpha})\geq \alpha n$, let $U$ be as chosen in 
Lemma~\ref{strong-connect} with $G=G_{\alpha}$ and $d=\frac{p\alpha^2n^2}{16}$, and $H_1,\dots,H_q$ be the connected components of $G_{\alpha}[V(G_{\alpha})\backslash U]$. Since $|V(H_i)|\geq \frac{\alpha n}{2}$ holds for each $1\leq i\leq q$, we have $q\leq \frac{2}{\alpha}$.

Let $\mathcal{P}$ be the event that the bipartite subgraph of $H$ induced on $V(H_i)\cup V(H_j)$ contains a matching of size $\frac{p\alpha^2n^2}{8}$, for each $1\leq i<j\leq q$. 

{\bf Claim} $\mathcal{P}$ happens with high probability. 

\begin{proof}[Proof of Claim]
Let 
$$S_1=\{v\in V(H_i):N(v)\cap H_j\neq \emptyset\},$$
$$S_2=\{(u,v):u,v\in V(H_i),N(v)\cap N(u)\cap H_j\neq \emptyset\}.$$
Recall that $|V(H_i)|\geq \frac{\alpha n}{2},|V(H_j)|\geq \frac{\alpha n}{2},p=o(\frac1n)$, we have 
$$\mathbb{E}(|S_1|)=|V(H_i)|(1-(1-p)^{|V(H_j)|})=(1+o(1))p|V(H_i)||V(H_j)|,$$
$$\mathbb{E}(|S_2|)={{|V(H_i)|}\choose{2}}(1-(1-p^2)^{|V(H_j)|})=(\frac12+o(1))p^2|V(H_i)|^2|V(H_j)|.$$
Let $\mathcal{P}_1$ denote the event that $|S_1|\geq \frac15p\alpha^2n^2$, and let $\mathcal{P}_2$ denote the event that $|S_2|\leq \frac{1}{20}p\alpha^2n^2$.
By Chernoff bound, 
$$\mathbb{P}[\text{the complement of }\mathcal{P}_1]\leq \mathbb{P}[|S_1|\leq \frac 45p|V(H_i)||V(H_j)|]=e^{-\Omega(pn^2)}=o(1).$$
By Markov's inequality,
$$\mathbb{P}[\text{the complement of }\mathcal{P}_2]\leq \mathbb{P}[|S_2|\geq \frac{\alpha^2}{20pn}\mathbb{E}(|S_2|)]\leq\frac{20pn}{\alpha^2}=o(1).$$
Conditioned on the events $\mathcal{P}_1$ and $\mathcal{P}_2$, which hold {\bf whp}, we have $|S_1|\ge \frac15 p\alpha^2 n^2$ and $|S_2|\le \frac1{20} p\alpha^2 n^2$. 
To construct a matching, define the \emph{conflict graph} $\Gamma$ on vertex set $S_1$ by connecting $u,v\in S_1$ with an edge if $u$ and $v$ share a common neighbor in $V(H_j)$, i.e., $\{u,v\}\in S_2$. Then $|E(\Gamma)|=|S_2|\le \frac1{20}p\alpha^2 n^2$.

By Turán's theorem (or a simple greedy argument), $\Gamma$ contains an independent set $S'\subseteq S_1$ of size
\[
|S'|\ge \frac{|S_1|^2}{|S_1| + 2|E(\Gamma)|} 
\ge \frac{(\frac15 p\alpha^2 n^2)^2}{\frac15 p\alpha^2 n^2 + 2\cdot \frac1{20} p\alpha^2 n^2} 
= \frac{2}{15} p\alpha^2 n^2 > \frac{p\alpha^2 n^2}{8}.
\]

By construction, no two vertices in $S'$ share a neighbor in $V(H_j)$. Hence we can greedily select one distinct neighbor in $V(H_j)$ for each vertex in $S'$ to form a matching of size at least $|S'|\ge \frac{p\alpha^2 n^2}{8}$ between $V(H_i)$ and $V(H_j)$, as desired. Applying a union bound over all  $(i,j)$ pairs, $\mathcal{P}$ happens with high probability.
\end{proof}


Next we prove that $s(H)\geq \frac{p\alpha^2n^2}{32}$. Assume to the contrary that there exists a partition $V(H)=W_0\cup W_1\cup W_2$, such that $|V(W_0)|<\frac{p\alpha^2n^2}{32}=\frac d2$, $|W_1|=|W_2|>|U|$, and that $E(W_1,W_2)=\emptyset$. We prove that there exists some $1\leq i,j\leq q$ such that $V(H_i)\subset W_1\cup W_0, V(H_j)\subset W_2\cup W_0$. Since $|W_1|>|U|$, $W_1$ must intersect with some $V(H_i)$. If $V(H_i)\cap W_2\neq \emptyset$, then we can break $H_i$ into two connected components by deleting $H_i\cap W_0$, which contains no more than $\frac d2$ vertices, a contradiction with the fact that $H_i$ is $d$-connected. Hence $V(H_i)\subset W_0\cup W_1$, and similarly there exists some $j$ such that $V(H_j)\subset W_0\cup W_2$. By $\mathcal{P}$ there is a matching $M$ of edge size $2d$, and $E(M)\subset E(V(H_i), V(H_j))$. However, at most $\frac d2$ edges of $M$ intersect with $W_0$, and the remaining edges intersect with both $W_1$ and $W_2$. Consequently, $E(W_1, W_2)$ is non-empty, a contradiction. By Lemma~\ref{3-connected}, $H$ is 3-connected. By Theorem~\ref{equipartition},  we can find the desired spanning tree with high discrepancy and we complete the proof. 
\end{proof}



\section*{Acknowledgement}
{Wenchong Chen thanks Lior Gishboliner for insightful discussions on relative topics. Additionally, he gratefully acknowledges the organizers of ECOPRO 2025 Student Research Program.
}

\bibliographystyle{abbrv}
\bibliography{main}

\end{document}